\def\On{\mbox{On}}

\def\mm{\mathfrak{M}}

\documentclass{article}
\usepackage{amsthm} 
\usepackage{graphicx} 
\usepackage{amssymb,latexsym,times}
\newtheorem{theorem}{Theorem}

\newtheorem{lemma}[theorem]{Lemma}
\newtheorem{proposition}[theorem]{Proposition}%

\begin{document}

\title{An extension of a theorem of Zermelo\footnote{I am indebted to  John Steel and Philip Welch for helpful discussions concerning this paper.}}
\author{Jouko V\"a\"an\"anen\\
Department of Mathematics and Statistics,\\ University of Helsinki\\
\and
Institute for Logic, Language and Computation,\\ University of Amsterdam}
\maketitle

\def\FUN{\mathrm{FUN}}
\def\ISM{\mathrm{ISM}}
\def\dom{\mathrm{dom}}
\def\ran{\mathrm{ran}}
\def\mm{\mathcal{M}}
\def\tr{\mathrm{tr}}
\def\TC{\mathrm{TC}}

\begin{abstract}
We show that if  $(M,\in_1,\in_2)$ satisfies the first order Zermelo-Fraenkel axioms of set theory when  the membership relation is $\in_1$ and also when the membership relation is $\in_2$, and in both cases the formulas are allowed to contain both $\in_1$ and $\in_2$, then $(M,\in_1)\cong (M,\in_2)$, and the isomorphism is definable in $(M,\in_1,\in_2)$. This extends Zermelo's 1930 theorem in \cite{zbMATH02562682}.
\end{abstract}

Zermelo \cite{zbMATH02562682} proved that if $(M,\in_1)$ and $(M,\in_2)$ both satisfy the second order Zermelo-Fraenkel axioms in which the Separation Schema and the Replacement Sche\-ma of $ZFC$ are replaced by single second order axioms, then $(M,\in_1)\cong(M,\in_2)$. We extend this as follows: Let us consider the vocabulary $\{\in_1,\in_2\}$, where both $\in_1$ and $\in_2$ are binary predicate symbols. Let $ZFC(\in_1)$ denote the {\em first order} Zermelo-Fraenkel axioms of set theory when $\in_1$ is the membership relation but formulas are allowed to contain $\in_2$ too. Similarly, in $ZFC(\in_2)$  the membership relation is $\in_2$  but formulas are allowed to contain $\in_1$ too.
We prove the following theorem:

\begin{theorem}\label{maintheorem}
If $(M,\in_1,\in_2)\models ZFC(\in_1)\cup ZFC(\in_2)$, then $(M,\in_1)\cong(M,\in_2)$ via a definable class function.
\end{theorem}

The result of Zermelo readily follows from our theorem. The important difference between our result and Zermelo's result is that our theories $ZFC(\in_1)$ and $ZFC(\in_2)$ are {\em first order} theories. It is important that we allow in these axiom systems formulas from the extended vocabulary $\{\in_1,\in_2\}$. Without this the result would be blatantly false as there are countable non-isomorphic models of $ZFC$, assuming there are models of $ZFC$ at all. Since the isomorphism in Theorem~\ref{maintheorem} is definable, the result can be seen as a provable theorem of the first order theory $ZFC(\in_1)\cup ZFC(\in_2)$.

Theorem~\ref{maintheorem} resembles the categoricity conclusion for set theory in \cite[page 18]{Martin}. There are two main differences: First, the axiomatization of set theory in \cite{Martin} is informal, based on the Axiom of Extensionality and an informal full Comprehension Axiom, while our result is completely formal and in the context of ZFC. Secondly, it is assumed in \cite{Martin} that the two $\in$-relations give rise to the same (informal) structure of the ordinals, owing to the uniqueness of the ordinal concept. We do not make this assumption but rather {\em prove} that the two $\in$-relations have isomorphic ordinals. Martin's work has been extended to class theory in \cite{welch}. Theorem~\ref{maintheorem} was stated without proof in \cite[page 104]{Vaananen2012-VNNSOL}.  

We call our theorem an {\em internal categoricity} result because it shows that one cannot have in one and the same domain two non-isomorphic membership-relations $\in_1$ and $\in_2$ if these relations can ``talk'' about each other. 

Our theorem is a strong robustness result for set theory. Essentially, the model cannot be changed ``internally''. To get a non-isomorphic model one has to go ``outside'' the model. Such robustness is important for set theory because set theory is already the ``outside'' of mathematics, the framework where mathematics is (or can be) built. 

How are the numerous independence results in harmony with this internal categoricity? Let us take the Continuum Hypothesis CH as an example. CH is independent of $ZFC$ in the sense that both $ZFC\cup\{CH\}$ and $ZFC\cup\{\neg CH\}$ are consistent, if $ZFC$ itself is. Internal categoricity means simply that if $(M,\in_1)$ satisfies $CH$ and $(M,\in_2)$ satisfies $\neg CH$, then either $(M,\in_1)$ or $(M,\in_2)$ does not satisfy the Separation Schema or the Replacement Schema if formulas are allowed to mention the other membership-relation.  Such models cannot be internal to each other in the sense discussed.

In the below proof we will work in $ZFC(\in_1)\cup ZFC(\in_2)$ but in fact operate all the time in either $\in_1$-set theory or in $\in_2$-set theory. We have to keep the two set theories separate even though they also interact via the Separation and Replacement Schemas in the joint vocabulary $\{\in_1,\in_2\}$.

Let $\tr_i(x)$ be the formula $\forall t\in_i x\forall w\in_i t(w\in_1 x)$. Let $\TC_i(x)$ be the  unique $u$ such that $\tr_i(u)\wedge \forall v\in_ix(v\in_iu)\wedge\forall v((\tr_i(v)\wedge \forall w\in_ix(w\in_iv))\to\forall w\in_iu(w\in_i v))$ (``$u$ is the $\in_i$-transitive closure of $x$''). When we write $\TC_i(\{x\})$, we mean by $\{x\}$ the singleton $\{x\}$ in the sense of $\in_i$. Let $\phi(x,y)$ be the formula $\exists f\psi(x,y,f)$, where $\psi(x,y,f)$ is the conjunction of the following formulas (where  $f(t)$, $f(w)$ and $f(x)$ are in the sense of $\in_1$):

\begin{description}
\item [(i)] In the sense of $\in_1$, the set $f$ is a function with $\TC_1(\{x\})$ as its domain.
\item [(ii)] $\forall t\in_1 \TC_1(x)(f(t)\in_2 \TC_2(y))$ 
\item [(iii)] $\forall t\in_2 \TC_2(y)\exists w\in_1 \TC_1(x)(t=f(w))$
\item [(iv)] $\forall t\in_1 \TC_1(x)\forall w\in_1 \TC_1(\{x\})(t\in_1w\leftrightarrow f(t)\in_2 f(w))$
\item [(v)] $f(x)=y$ 
\end{description} 



We prove a sequence of lemmas about the formulas $\phi(x,y)$ and $\psi(x,y,f)$:

\begin{lemma} If $\psi(x,y,f)$ and $\psi(x,y,f')$, then  $f=f'$. 
\end{lemma}

\begin{proof}
To prove $f=f'$ assume  $w\in_1 \TC_1(\{x\})$. We show $f(w)=f'(w)$. W.l.o.g. $f(s)=f'(s)$ for $s\in_1 w$. Suppose $t\in_2 f(w)$. Clearly, $t\in_2 \TC_2(\{y\})$. By (iii), $t=f(s)$ for some $s\in_1 \TC_1(\{x\})$. By (iv), $s\in_1 w$. By (iv) again, $f'(s)\in_2 f'(w)$. By assumption, $f(s)=f'(s)$. Hence $t\in_2 f'(w)$. Thus $\forall t(t\in_2 f(w)\to t\in_2 f'(w))$. By symmetry, $f(w)=f'(w)$.  
\end{proof}

\begin{lemma}\label{fds}\label{fdsy} Suppose $\psi(x,y,f)$. If $x'\in_1x$, then $\phi(x',f(x'))$. 
 If  $y'\in_2y$, then there is $x'\in_1x$ such that $f(x')=y'$ and $\phi(x',y')$. 
\end{lemma}

\begin{proof}
Let $y'=f(x')$ and $f'=f\restriction \TC_1(\{x'\})$. Clearly now $\psi(x',y',f')$. Hence $\phi(x',f(x'))$.
The other claim is proved similarly. \end{proof}

\begin{lemma}\label{L5}\label{L6}
If $\phi(x,y)$ and $\phi(x,y')$, then $y=y'$. If $\phi(x,y)$ and $\phi(x',y)$, then $x=x'$.
\end{lemma}

\begin{proof} We may assume the claim holds for all $\in_1$-elements of $x$.
Suppose $\psi(x,y,f)$ and $\psi(x,y',f')$.
We prove $y=y'$. Let $s\in_2 y$. By Lemma~\ref{fdsy} there is $t\in_1 x$ such that $f(t)=s$ and $\phi(t,s)$. By (iv), $s\in_1 x$. Let $s'=f'(t)$. By (iv),  $s'\in_2 y'$. By Lemma~\ref{fds} again, $\phi(t,s')$. By the Induction Hypothesis, $s=s'$. We have proved $\forall s(s\in_2 y\to s\in_2 y')$. The converse follows from symmetry. Now to the second claim. We may assume the claim holds for all $\in_2$-elements of $y$.
Suppose $\psi(x,y,f)$ and $\psi(x',y,f')$. 
We prove $x=x'$. Let $s\in_1 x$. Thus $f(s)\in_2 y$. There is $s'\in_1 \TC_1(\{x'\})$ such that $f'(s')=f(s)$. Now $\phi(s,f(s))$ and $\phi(s',f(s))$ by Lemma~\ref{fds}. Since $f(s)\in_2 y$, $s=s'$. Hence $s\in_1 x'$. We have proved $\forall s(s\in_1 x\to s\in_1 x')$. The converse follows from symmetry.\end{proof}

\begin{lemma}\label{L7}
If $\phi(x,y)$ and $\phi(x',y')$, then $x\in_1x'\leftrightarrow y\in_2y'$.
\end{lemma}

\begin{proof} 
Suppose $\psi(x,y,f)$ and $\psi(x',y',f')$. Suppose $x\in_1 x'$. Then $z=f'(x)\in_2 y'$. By Lemma~\ref{fds}, $\phi(x,z)$. We have $\phi(x,y)$ and $\phi(x,z)$. By Lemma~\ref{L5}, $y=z$. Hence $y\in_2 y'$. The converse is similar. 
\end{proof}

Let $\On_1(x)$ be the $\in_1$-formula saying that $x$ is an ordinal i.e. a transitive set of transitive sets, and similarly $\On_2(x)$. For $\On_1(\alpha)$ let $V^1_\alpha$ be the $\alpha^{th}$ level of the cumulative hierarchy in the sense of $\in_1$, and similarly $V^2_y$ when $\On_2(y)$.

\begin{lemma}\label{L8}
If $\phi(\alpha,y)$, then $\On_1(\alpha)$ if and only if  $\On_2(y)$. If $\alpha$ is a limit ordinal then so is $y$ i.e. if $\forall u\in_1 \alpha\exists v\in_1\alpha(u\in_1 v)$, then $\forall u\in_2 y\exists v\in_2 y(u\in_2 v)$, and vice versa. 
\end{lemma}

\begin{proof} Let us fix $y$.
Suppose $\psi(\alpha,y,f)$.  We prove that $y$ is a transitive set of transitive sets. Suppose $w\in_2s\in_2y$. There are  $t\in_1 \alpha$ and $u\in_1 t$ such that $f(t)=s$ and $f(u)=w$. Now $w\in_2 y$ follows from $u\in_1\alpha$. This shows that $\tr_2(y)$. Similarly one proves that all $s\in_2y$ satisfy $\tr_2(y)$. This ends the proof of the first claim. The second claim is proved similarly.
\end{proof}

\begin{lemma}\label{L9}
Suppose $\psi(\alpha,y,f)$. If  $\On_1(\alpha)$ (or equivalently $\On_2(y)$), then there is $\bar{f}\supseteq f$ such that  $\psi(V^1_\alpha,V^2_y,\bar{f})$. 
\end{lemma}

\begin{proof}
We use induction on $\alpha$. Suppose the claim holds for $\alpha$. We prove the claim for $\alpha+1$. Suppose to this end  $\psi(\alpha+1,y+1,f)$. We construct $\bar{f}$ such that $\psi(V^1_{\alpha+1},V^2_{y+1},\bar{f})$. From $\psi(\alpha+1,y+1,f)$ we obtain $\psi(\alpha,y,f\restriction \alpha)$. By assumption there is $g\supseteq f\restriction\alpha$ such that $\psi(V^1_\alpha,V^2_y,g)$.  Let $\theta(u,v)$ be the formula
$$\forall w(w\in_2 v\leftrightarrow (w\in_2 V^2_y\wedge\exists t\in_1 u(g(t)=w))).$$
It follows from the Separation Schema of $ZFC(\in_2)$ that for all $u\in_1V^1_{\alpha+1}$ there is $v$ such that $\theta(u,v)$. By the Replacement Schema of $ZFC(\in_1)$, we can let $\bar{f}$ be a function (in the sense of $\{\in_1\}$) such that for all $u\in_1V^1_{\alpha+1}$ we have  $\theta(u,\bar{f}(u))$. It is easy to see, using the Separation Schema of $ZFC(\in_1)$, that $\psi(V^1_{\alpha+1},V^2_{y+1},\bar{f})$.

Suppose then the claim holds for all $\beta<\alpha=\bigcup\alpha$. For each $\beta<\alpha$ there is thus some $g_\beta$ such that $\psi(V^1_\beta,V^2_{f(\beta)},g_\beta)$.  By the Replacement Schema of $ZFC(\in_1)$ we can form the $\in_1$-set  $\bar{f}=\bigcup_{\beta<\alpha}g_\beta$. It is easy to see that $\psi(V^1_\alpha,V^2_y,\bar{f})$. 
\end{proof}

\begin{lemma}\label{L10}
$\forall x\exists y\phi(x,y)$ and $\forall y\exists x\phi(x,y)$. 
\end{lemma}

\begin{proof}
Let us first assume that both 
\begin{equation}\label{eq1}
\forall \alpha(\On_1(\alpha)\to\exists y\phi(\alpha,y))
\end{equation} and 
\begin{equation}\label{eq2}
\forall y(\On_2(y)\to\exists \alpha\phi(\alpha,y)).
\end{equation} hold.
In order to prove $\forall x\exists y\phi(x,y)$, suppose $x$ is given. There is $\alpha$ such that $\On_1(\alpha)$ and $x\in_1 V^1_\alpha$. By (\ref{eq1}) there are $v$ and $f$ such that $\psi(\alpha,v,f)$. By Lemma~\ref{L9} there is $\bar{f}\supseteq f$ such that $\psi(V^1_\alpha, V^2_v,\bar{f})$.  By Lemma~\ref{fds}, $\phi(x,\bar{f}(x))$. Thus $\exists y\phi(x,y)$. 

In order to prove $\forall y\exists x\phi(x,y)$, suppose $y$ is given. There is $v$ such that $\On_2(v)$ and $y\in_2 V^2_v$. By (\ref{eq2}) there are $\alpha$ and $f$ such that $\psi(\alpha,v,f)$. By Lemma~\ref{L9} there is $\bar{f}\supseteq f$ such that $\psi(V^1_\alpha, V^2_v)$.  By condition (iii) of the definition of $\psi$ there is $w\in_1 V^1_\alpha$ such that $\bar{f}(w)=y$. By Lemma~\ref{fds}, $\phi(w,\bar{f}(w))$. Thus $\exists x\phi(x,y)$.  

Thus it suffices to show that the failure of (\ref{eq1}) or (\ref{eq2}) to hold leads to a contradiction.

\medskip

\noindent{\bf Case 1:} $\neg$(\ref{eq1})$\wedge\neg$(\ref{eq2}). Let $\alpha$ be the $\in_1$-least $\alpha$ such that $\On_1(\alpha)\wedge\neg\exists y\phi(\alpha,y)$.
Let $y$ be the $\in_2$-least $y$ such that $\On_2(y)\wedge\neg\exists \beta\phi(\beta,y)$. It is easy to see that $\phi(\alpha,y)$, a contradiction.

\medskip

\noindent{\bf Case 2:} (\ref{eq1})$\wedge\neg$(\ref{eq2}).  
Let $y$ be the $\in_2$-least $y$ such that $\On_2(y)\wedge\neg\exists \alpha\phi(\alpha,y)$. Now, $\forall t\in_2 y\exists \alpha(\On_1(\alpha)\wedge\phi(\alpha,t))$. Clearly, $y$ is an $\in_2$-limit ordinal. Suppose $z\in_2 V^2_t$, where $t\in_2 y$. Let $\alpha$ and $f$ be such that $\On_1(\alpha)\wedge\psi(\alpha,t,f)$.  By Lemma~\ref{L9} there is $\bar{f}\supseteq f$ such that $\psi(V^1_\alpha, V^2_t,\bar{f})$. There is $x\in_1 V^1_\alpha$ such that $\bar{f}(x)=z$. Thus $\phi(x,z)$ and hence
$$\forall z\in_2 V^2_y\exists x\ \phi(x,z).$$ By the Replacement Schema in $ZFC(\in_2)$ there is $c$ such that  
\begin{equation}\label{a}
\forall z\in_2 V^2_y\exists x\in_2 c\ \phi(x,z).
\end{equation}
Let $\alpha$ be such that $c\in_1 V^1_\alpha$. By (\ref{eq1}) there are $t$ and $f$ such that $\phi(\alpha,t,f)$.  Necessarily, $t\in_2 y$. By Lemma~\ref{L9} there is $\bar{f}\supseteq f$ such that $\psi(V^1_\alpha, V^2_t,\bar{f})$. In particular, $\bar{f}(c)\in_2V^2_y$. By (\ref{a}) there is $b\in_2 c$ such that $\phi(b,\bar{f}(c))$. Since also $\phi(c,\bar{f}(c))$, Lemma~\ref{L6} gives $c=b$. Thus $c\in_2 c$, a contradiction.
\medskip

\noindent{\bf Case 3:} $\neg$(\ref{eq1})$\wedge$(\ref{eq2}). This case is analogous to Case 2.

\medskip

\end{proof}

\begin{proposition}\label{main}
The class defined by $\phi(x,y)$ is an isomorphism between  the $\in_1$-reduct and the $\in_2$-reduct.
\end{proposition}

\begin{proof}
By Lemmas~\ref{L5}, \ref{L7} and \ref{L10}.
\end{proof}

A similar result holds for first order Peano arithmetic, extending the categoricity result of Dedekind \cite{02693454} of  second order Peano arithmetic. The proof (see \cite{mapping}) of this is similar, but somewhat easier.



\end{document}